\documentclass[11pt]{amsart}

\usepackage{amscd,amssymb,amsopn,amsmath,amsthm,mathrsfs,graphics,amsfonts,enumerate,verbatim,calc
}

\usepackage[all]{xy}

\usepackage{color}

\usepackage[OT2,OT1]{fontenc}
\newcommand\cyr{%
\renewcommand\rmdefault{wncyr}%
\renewcommand\sfdefault{wncyss}%
\renewcommand\encodingdefault{OT2}%
\normalfont
\selectfont}
\DeclareTextFontCommand{\textcyr}{\cyr}

\usepackage{amssymb,amsmath}

\DeclareFontFamily{OT1}{rsfs}{}
\DeclareFontShape{OT1}{rsfs}{n}{it}{<-> rsfs10}{}
\DeclareMathAlphabet{\mathscr}{OT1}{rsfs}{n}{it}

\numberwithin{equation}{section}
\newtheorem{theorem}{Theorem}[section]
\newtheorem{lemma}[theorem]{Lemma}

\newtheorem{corollary}[theorem]{Corollary}

\newtheorem*{maintheorem}{Main Theorem}

\theoremstyle{definition}
\newtheorem{definition}[theorem]{Definition}
\newtheorem{remark}[theorem]{Remark}
\theoremstyle{remark}

\newtheorem{example}[theorem]{Example}

\newcommand{\Ass}{\operatorname{Ass}}

\newcommand{\Spec}{\operatorname{Spec}}

\newcommand{\Ht}{\operatorname{ht}}

\newcommand{\depth}{\operatorname{depth}}

\newcommand{\Sing}{\operatorname{Sing}}

\newcommand{\fp}{\frak{p}}

\newcommand{\fa}{\frak{a}}

\begin{document}
\title[On the associated primes of local cohomology]{On the associated primes of local cohomology }

\author{Hailong Dao}
\address{Department of Mathematics\\
University of Kansas\\
 Lawrence, KS 66045-7523 USA}
\email{hdao@ku.edu}

\author[Pham Hung Quy]{Pham Hung Quy}
\address{Department of Mathematics, FPT University, Hoa Lac Hi-Tech Park, Ha Noi, Viet Nam}
\email{quyph@fpt.edu.vn}

\thanks{
P.H. Quy is partially supported by a fund of Vietnam National Foundation for Science
and Technology Development (NAFOSTED) under grant number
101.04-2014.25. This paper was done while the second author was visiting Vietnam Institute for Advanced Study in Mathematics. He would like to thank the VIASM for hospitality and financial support.}



\begin{abstract} Let $R$ be a commutative Noetherian ring of prime characteristic $p$. In this paper we give a short proof using filter regular sequences  that the set of associated prime ideals of $H^t_I(R)$ is finite for any ideal $I$ and for any $t \ge 0$ when $R$ has finite $F$-representation type or finite singular locus. This  extends a previous result by Takagi-Takahashi and  gives affirmative answers for a problem of Huneke in many new classes of rings  in positive characteristic. We also give a criterion about the singularities of $R$ (in any characteristic) to guarantee that the set  $\Ass H^2_I(R)$ is always finite.  
\end{abstract}

\maketitle


\section{Introduction}
Throughout this paper, let $R$ be a commutative Noetherian ring and $I$ an ideal of $R$. Local cohomology,  introduced by Grothendieck, is an important tool in both algebraic geometry and commutative algebra (cf. \cite{BS98}). In general, the local cohomology module $H^t_I(R)$ is not finitely generated as an $R$-module. Therefore an important problem is identifying finiteness properties of local cohomology. Huneke in \cite{Hu92} raised the question of whether local cohomology modules of Noetherian rings always have finitely many associated primes.  This problem is an active research area in commutative algebra in the last two decades. 

Nowadays, we known that Huneke's question has a negative answer in general; many interesting counterexamples were constructed by
Katzman \cite{K02}, Singh \cite{Si00}, and Singh and Swanson \cite{SS04}, even for a hypersurface $R$ with ``good" singularities. On the other hand,  affirmative results have been obtained for nice classes of $R$, typically with small singular loci.  In particular Huneke's question has been largely settled when $R$ is regular (cf. \cite{BBL14} \cite{HS93}, \cite{L93} and \cite{L97}), although important cases remain open, see Remark \ref{localize}. Many partial affirmative results for singular $R$ have also been proved, for example see \cite{BF00}, \cite{M01}, \cite{TT08} and \cite{Q13}. 

In this paper we first consider  Huneke's  problem  when $R$ contains a field of prime characteristic $p$. In this situation perhaps the most significant result to date in the singular situation came from the work of Takagi and Takahashi (\cite{TT08}), who  showed that Huneke's question has an affirmative answer when $R$ is a Gorenstein ring with finite $F$-representation type.  The notion of finite $F$-representation type was introduced by Smith and Van den Bergh in \cite{SV97} as a characteristic $p$ analogue of the notion of finite representation type. One part of our main theorem extends their result by dropping the Gorenstein condition. 

\begin{maintheorem} Let $R$ be a Noetherian ring of prime characteristic $p$ that has finite $F$-representation type or finite singular locus. Then $H^t_I(R)$ has only finitely many associated prime ideals for any ideal $I$ and for any $t \ge 0$. Consequently, $H^t_I(R)$ has closed support for any ideal $I$ and for any $t \ge 0$.
\end{maintheorem}

\begin{remark}
We have been informed that Mel Hochster and Luis N\'u\~nez-Betancourt have obtained the same result above with different method in \cite{HN15}. 
\end{remark}

Our main result yields affirmative answers of Huneke's problem for many important classes of rings. The proof is rather short, the main ingredient being the theory of filter regular sequence and an extremely useful isomorphism of local cohomology modules by Nagel-Schenzel, see Lemma \ref{nagel-schenzel}. 

In the last Section, we use the same approach to study the problem of when $\Ass H^2_I(R)$  is always finite and give a sufficient condition on the singularities of $R$ in any characteristic, see Theorem \ref{H2}.

\section{Results in positive characteristic}
In this Section we always assume that $R$ contains a field of prime characteristic $p>0$. For such a ring, we have the Frobenius homomorphism $F: R \to R; x \mapsto x^p$. We also denote $F^e = F \circ \cdots \circ F$ (iterated $e$ times) for all $e \ge 0$. For any $R$-module $M$, we
denote by ${}^eM$ the module $M$ with its $R$-module structure pulled back via $F^e$. That is, ${}^eM$ is just $M$ as an abelian group, but its
$R$-module structure is determined by $r \cdot m = r^{p^e}m$ for all $r \in R$ and $m \in M$. We say $R$ is $F$-finite if ${}^1R$ is a finitely generated $R$-module.  Let $I$ be an ideal of $R$. For each $q = p^e$ we denote $I^{[q]} = (a^q \,|\, a\in I)$. It is easy to see that
$$R/I \otimes_R {}^eM \cong {}^eM/(I \cdot {}^eM)  \cong {}^e(M/I^{[q]}M).$$

The following is useful in the sequel (cf. \cite[Lemma 3.1]{TT08}).

\begin{lemma} \label{L2.1}
Suppose that $R$ is a ring of prime characteristic. Let $M$ be an $R$-module. Then $\mathrm{Ass}_RM = \mathrm{Ass}_R {}^eM$ for all $e \ge 1$.

\end{lemma}

Rings with finite $F$-representation type were first introduced by Smith and Van
den Bergh in \cite{SV97}, under the assumption that the Krull-Schmidt theorem holds for
them. Yao \cite{Y05} studied these rings in a more general setting.

\begin{definition}
  Let $R$ be a Noetherian ring of prime characteristic $p$. We say that $R$ has
{\it finite $F$-representation type} by finitely generated $R$-modules $M_1, \ldots, M_s$ if for every $e \ge 0$, the $R$-module ${^e}R$ is isomorphic to a finite direct sum of the $R$-modules $M_1, \ldots, M_s$, that is, there exists non-negative integers $n_{e1}, \ldots, n_{es}$ such that
$${}^eR \cong  \bigoplus_{i =1}^s M_i^{n_{ei}}.$$
We simply say that $R$ has finite $F$-representation type if there exist finitely
generated $R$-modules $M_1, \ldots, M_s$ by which $R$ has finite $F$-representation type.
\end{definition}

It is clear that if the ring $R$ has finite $F$-representation type, then $R$ is $F$-finite. We collect here some examples of rings with finite $F$-representation type. For the details see \cite[Example 1.3]{TT08}.

\begin{example} \label{exam}
\begin{enumerate}[{(i)}]
\item Let $R$ be an $F$-finite regular local ring of characteristic $p>0$ (resp. a polynomial ring $k[X_1, \ldots, X_n]$ over a field $k$ of characteristic $p>0$ such that $[k : k^p]< \infty$). Then $R$ has finite $F$-representation type  by the $R$-module $R$.

\item Let $R$ be a Cohen-Macaulay local ring of prime characteristic $p$ with finite representation type. Then $R$ has finite $F$-representation type.

\item Let $S = k[X_1, \ldots, X_n]$ be a polynomial ring over a field $k$ of characteristic $p>0$ such that $[k : k^p]< \infty$ and $\fa$ a monomial ideal of $S$. Then the quotient ring $R = S/\fa$ has finite $F$-representation type.

\item Let $R \hookrightarrow S$ be a finite local
homomorphism of Noetherian local rings of prime characteristic $p$ such that
$R$ is an $R$-module direct summand of $S$. If $S$ has finite $F$-representation type, so does $R$.

\item A normal semigroup ring over a field $k$ of characteristic $p>0$ such that $[k : k^p]< \infty$ has finite $F$-representation type. Also, rings of invariants of linearly reductive groups over such field have finite $F$-representation type.

\end{enumerate}

\end{example}

\begin{remark} It is easy to see that if $R$ has finite $F$-representation type, then so do $R[X]$ and $R[[X]]$, as well as localizations and completions of $R$. Combining this fact and the above examples we can produce many interesting examples of rings with finite $F$-representation type.
\end{remark}

We next recall the notion of $I$-filter regular sequence of
$R$ and its relation with local cohomology.
\begin{definition} Let $I$ be an ideal of a Noetherian ring $R$. We say a sequence $x_1,\ldots,x_t$ of elements contained in $I$ is an
{\it $I$-filter regular sequence of $R$} if
$$\mathrm{Supp}\,
((x_1,\ldots,x_{i-1}):x_i)/(x_1,\ldots,x_{i-1}) \subseteq V(I)$$
for all $i = 1,\ldots,t$, where $V(I)$ denotes the set of prime
ideals containing $I$. This condition is equivalent to  $x_i \notin \fp$ for all $\fp \in \mathrm{Ass}_R R/(x_1, \ldots, x_{i-1}) \setminus V(I)$ and for all $i = 1, \ldots, t$.
\end{definition}

The following is very useful to analyze the local cohomology modules by filter regular sequences (cf. \cite[Proposition 3.4]{NS95} for the original version when $I$ is the maximal ideal, and \cite[Proposition 2.3]{AS03} for the general version below). More applications can be found in \cite{QS16}.

\begin{lemma}[Nagel-Schenzel's isomorphism] \label{nagel-schenzel}Let $I$ be an ideal of a Noetherian ring $R$ and $x_1,\ldots,x_t$ an
$I$-filter regular sequence of $R$. Then we have
\[
H^i_{I}(R) \cong
\begin{cases}
H^i_{(x_1,\ldots, x_t)}(R) &\text{ if } i<t\\
H^{i-t}_{I}(H^t_{(x_1, \ldots, x_t)}(R)) &\text{ if } i \ge t.
\end{cases}
\]

\end{lemma}

\begin{remark}\label{filter}
It should be noted that the notion of filter regular sequence and the Nagel-Schenzel isomorphism are not dependent on the characteristic of the ring. Moreover, for any $t \ge 1$ we always can choose a $I$-filter regular sequence of $x_1, \ldots, x_t$. Indeed, by the prime avoidance lemma we can choose $x_1 \in I$ and $x_1 \notin \fp$ for all $\fp \in \mathrm{Ass}_RR \setminus V(I)$. For $i>1$ assume that we have $x_1, \ldots, x_{i-1}$, then we choose  $x_i \in I$ and $x_i \notin \fp$ for all $\fp \in \mathrm{Ass}_RR/(x_1, \ldots, x_{i-1}) \setminus V(I)$ by the prime avoidance lemma again. For more details, see \cite[Section 2]{AS03}. 
\end{remark}

\begin{lemma}\label{sing}
Let $R$ be a commutative Noetherian rings of characteristic $p>0$ and $\fa$ is an ideal. Then $$\cup_{e\geq 0} \Ass  R/\fa^{[p^e]} \subseteq \Ass R/\fa \cup\Sing(R).$$

\end{lemma}

\begin{proof}
Let $\fp \in \Ass R/\fa^{[q]}$ for some $q=p^e$. If $\fp \in \Ass R/\fa$, we are done. If not, we need to show that $\fp \in \Sing(R)$. Suppose not, then $R_{\fp}$ is regular. As Frobenius localizes and is flat over a regular ring, we have that $\depth (R/\fa^{[q]})_{\fp} = \depth (R/\fa)_{\fp}>0$, a contradiction. 
\end{proof}

We are now ready to prove the Main Theorem which we mentioned in the introduction.

\begin{proof}[Proof of Main Theorem] For any ideal $I$ and any $t \ge 0$ we prove that $\mathrm{Ass}_RH^i_I(R)$ is finite provided $R$ has finite $F$-representation type. The case $t=0$ is trivial. For $t\ge 1$, by Remark \ref{filter}  we can choose an $I$-filter regular sequence $x_1, \ldots, x_t$ of $R$. Set $\fa = (x_1, \ldots, x_t)$, by Lemma \ref{nagel-schenzel} we have
$$H^t_I(R) \cong H^0_I(H^t_{\fa}(R)).$$

Therefore $\mathrm{Ass}_RH^t_I(R) \subseteq \mathrm{Ass}_R H^t_{\fa}(R)$. Consider local cohomology as the direct limit of Koszul cohomologies we have
$$H^t_{\fa}(R) \cong \lim_{\overset{\longrightarrow}{n \in \mathbb{N}}} R/(x_1^n, \ldots, x_t^n).$$
In the case our ring is of prime characteristic $p>0$ we have the following isomorphism
$$H^t_{\fa}(R) \cong \lim_{\overset{\longrightarrow}{q}} R/(x_1^q, \ldots, x_t^q),$$
where $q = p^e$, $e \ge 0$. It is not hard to check that
$$\mathrm{Ass}_R H^t_{\fa}(R) \subseteq \bigcup_{q} \mathrm{Ass}_R R/(x_1^q, \ldots, x_t^q). \  \ \  (*)$$
Hence it is enough to prove that $\cup_{q} \mathrm{Ass}_R R/\fa^{[q]}$ is a finite set. On the other hand, let $M_1, \ldots, M_s$ be the finitely generated $R$-modules by which $R$ has finite $F$-representation type. For all $e \ge 0$ we have
$${}^e(R/\fa^{[q]}) \cong R/\fa \otimes_R {}^eR \cong R/\fa \otimes_R \big(\bigoplus_{i =1}^s M_i^{n_{ei}} \big) \cong \bigoplus_{i =1}^s (M_i/\fa M_i)^{n_{ei}}.$$
By Lemma \ref{L2.1} we have
$$\mathrm{Ass}_R R/\fa^{[q]} = \mathrm{Ass}_R {^e}(R/\fa^{[q]}) \subseteq \bigcup_{i=1}^s \mathrm{Ass}_R M_i/\fa M_i$$
for all $q = p^e$. Since $M_i$ is finitely generated for all $1 \le i \le s$, the proof is complete.

When $\Sing(R)$ is finite, the proof is the same until $(*)$, then we can invoke Lemma \ref{sing}.

\end{proof}

\begin{remark} \label{localize} We can also prove the assertion of our Theorem when $\Sing(R)$ is finite  by the localization at countably infinitely many primes technique in \cite{BQ16}. To prove that assume $\mathrm{Ass}_RH^t_I(R)$ is an infinite set for some ideal $I$ and some $t \ge 0$. Since $\mathrm{Sing}(R)$ is finite, we can choose a countably infinitely many set $\{\fp_i\}_{i \in \mathbb{N}} \subseteq \mathrm{Ass}_RH^t_I(R)$ such that $R_{\fp_i}$ is regular for all $i \ge 1$. Using the localization at countably infinitely many primes technique (if necessary) as in \cite[Lemma 1.1]{BQ16} we obtain a regular ring $T$ of prime characteristic such that $\mathrm{Ass}_TH^t_{IT}(T)$ is infinite. This is a contradiction. This proof would work in all characteristic, but even in characteristic $0$ we do not know if finiteness of associated primes hold for all regular rings. 
\end{remark}

\begin{corollary}
Let $R\to S$ be a homomorphism of  commutative Noetherian rings of characteristic $p>0$ that splits as $R$-modules. Assume that  $\Sing(S)$ is a finite set. Then $H^t_I(R)$ has only finitely many associated prime ideals for any ideal $I$ and for any $t \ge 0$.
\end{corollary}

\begin{proof}
One can use our main Theorem together with \cite{NB12}.
\end{proof}

\section{On $\Ass H^2_I(R)$}

In this section we consider rings of any characteristic.  Despite all the positive and negative results we have mentioned, it remains mysterious to understand the associated primes of $H^2_I(R)$, the first non-trivial case. It can be infinite even for a hypersurface domain (\cite[Remark 4.2]{SS04}). On the other hand, we do not know if this set is always finite if $R$ is an excellent  normal domain (however, it is finite if $\dim R\leq 4$, \cite{M01, HKM09}). In this section, we analyze this question a bit further and identify an interesting condition that guarantees the finiteness of $\Ass H^2_I(R)$. 

\begin{definition}
A local normal domain $R$ is said to satisfy condition $(D3)$ if for any reflexive ideal $I$, $\depth I\geq \min \{\dim R, 3\}$. 
\end{definition}

\begin{remark}\label{h2ex}
The following local rings satisfy condition $(D3)$: any normal domain of dimension at most $2$,  any UFD of depth at least $3$. In particular, any complete intersection regular in codimension $3$ is $(D3)$. More interestingly, $R$ is $(D3)$ if it is $\mathbb Q$-factorial and is strongly $F$-regular (positive characteristic)  or   has log-terminal singularities (characteristic $0$), see \cite[Theorem 3.1]{PS14}.
\end{remark}

\begin{theorem}\label{H2}
Suppose $R$ is a normal domain such that the set $$N = \{ \fp \in \Spec R | \  R_{\fp} \ \textit{is not}\  (D3)\}$$ is finite. Then $\Ass H^2_I(R)$ is finite for any ideal $I$.
\end{theorem}

\begin{proof}
We begin by claiming that the set $\{ \fp \in \Ass H^2_I(R), \Ht \fp \leq 2\}$ is finite. Obviously, for $\fp$ in such set, $\Ht \fp =2$. Then our claim follows from \cite[Theorem 3.6]{BQ16}. 

If $R$ is excellent  the claim also follows from the  proof of \cite[Corollary 2.8]{HKM09}. Let $\sqrt{I} = J\cap K$ where $\Ht I=1$, $\Ht K\geq 2$. If $K\subseteq \fp$, we are done. If not, then $H^2_{I_{\fp}}(R_{\fp}) =  H^2_{J_{\fp}}(R_{\fp})$. But as $R$ is an excellent normal domain, $R_{\fp}$ is analytically irreducible, so $H^2_{J_{\fp}}(R_{\fp})=0$ by the Hartshorne-Lichtenbaum Vanishing Theorem, a contradiction.

Next we show that the set $\{ \fp \in \Ass H^2_I(R), \Ht \fp \geq  3\}$ is a subset of $N$, finishing the proof. By the Nagel-Schelzen trick as in the proof of the main Theorem (or \cite[Proposition 2.7]{HKM09}), we can assume that $\fp \in \Ass R/(x^n,y^n)$ for some $n\geq 0$. We need to show $R_{\fp}$ is not $(D3)$. Suppose it is. There is an exact sequence:
$$0 \to R/(x^n:y^n) \to R/(x^n) \to R/(x^n,y^n) \to 0$$
By localizing and counting depth, it is clear that $(x^n:y^n)$ satisfies Serre's condition $(S_2)$, so it is a reflexive ideal. As $\dim R_{\fp}\geq 3$ and it is $(D3)$, it follows that $\depth R_{\fp}/(x^n:y^n)_{\fp}\geq 2$.   Counting depth again shows that $\depth R_{\fp}/(x^n,y^n)_{\fp}>0$, contradicting the fact that $\fp \in \Ass R/(x^n,y^n)$. 
\end{proof}

\begin{corollary}
Let $R$ be a commutative domain which is regular in codimension $3$ and is locally a complete intersection. Then $\Ass H^2_I(R)$ is finite for any ideal $I$ in $R$.
\end{corollary}

\begin{remark}\label{h2cor}
Under mild conditions, the properties that guarantee condition $(D3)$, as explained in Remark \ref{h2ex}, are open properties on $\Spec R$. 
Thus, in such situation Theorem \ref{H2} would yield statements such as ``if $R$ is a local normal domain which is  $\mathbb Q$-factorial and has log terminal singularities in codimension $\dim R-2$, then $\Ass H^2_I(R)$  is always finite".

\end{remark}

\begin{remark}
Singh and Swanson constructed (\cite[Theorem 5.1]{SS04}) a $F$-regular  hypersurface UFD $R$ and an ideal $I$ such that $\Ass H^3_I(R)$ is infinite. This shows that having good singularities alone may not help for higher local cohomology. However, in view of Theorem \ref{H2}, it is plausible that if $R$ is regular in codimension $c$ and has nice singularities, then $\Ass H^t_I(R)$ is finite for $t$ small relative to $c$.  The example in \cite{SS04} is normal but not $(R_2)$.
\end{remark}

\end{document}